\documentclass[a4paper,10pt]{amsart}
\usepackage{amsmath,amsfonts,amssymb,amsthm, amsfonts}
\usepackage[latin1]{inputenc}
\usepackage[english]{babel}

\usepackage[cmtip,arrow]{xy}
\usepackage{pb-diagram,pb-xy}

\newcommand{\PP}{\mathbb{P}}
\newcommand{\NN}{\mathbb{N}}
\newcommand{\ZZ}{\mathbb{Z}}
\newcommand{\QQ}{\mathbb{Q}}
\newcommand{\RR}{\mathbb{R}}
\newcommand{\CC}{\mathbb{C}}
\newcommand{\Res}{\begin{abstract}}

\theoremstyle{plain}
\newtheorem{thm}{Theorem}[section] 
\newtheorem{prop}[thm]{Proposition}
\newtheorem{lemma}[thm]{Lemma}
\newtheorem{cor}[thm]{Corollary}
\theoremstyle{remark}
\theoremstyle{definition}
\newtheorem{rem}[thm]{Remark}

\newtheorem{defi}[thm]{Definition}
\newtheorem{defis}[thm]{Definitions}
\newtheorem{recalls}[thm]{Recalls}
\newtheorem{notas}[thm]{Notations}

\title[Skew symmetric logarithms and geodesics on $O_n(\RR)$]{Skew symmetric logarithms and geodesics on $O_n(\RR)$}

\author[A. Dolcetti]{Alberto Dolcetti}
\address{Dipartimento di Matematica e Informatica ``Ulisse Dini''\\Viale Morgagni 67/a\\50134 Firenze, ITALIA}\email{alberto.dolcetti@unifi.it}

\author[D. Pertici]{Donato Pertici}
\address{Dipartimento di Matematica  e Informatica ``Ulisse Dini''\\Viale Morgagni 67/a\\50134 Firenze, ITALIA}
\email{donato.pertici@unifi.it}

\begin{document}

\selectlanguage{english}

\parindent 0pt

\begin{abstract}
We investigate  the connections between the differential-geometric properties of the exponential map from the space of real skew symmetric matrices onto the group of real special orthogonal matrices and the manifold of real orthogonal matrices equipped with the Riemannian structure induced by the Frobenius metric.
\end{abstract}

\maketitle

\tableofcontents

\renewcommand{\thefootnote}{\fnsymbol{footnote}}
\footnotetext{
This research was partially supported by MIUR-PRIN: ``Variet\`a reali e complesse: geo\-me\-tria, topologia e analisi armonica'' and by GNSAGA-INdAM.}
\renewcommand{\thefootnote}{\arabic{footnote}}
\setcounter{footnote}{0}

\medskip

{\scshape{Mathematics~Subject~Classification~(2010):}} 15A16, 53C22.

\section*{Introduction}

The \emph{exponential map} induces a surjection from the vector space $\mathcal{A}_n$ of real \emph{skew symmetric matrices} of order $n$ and the manifold $SO_n$ of real \emph{special orthogonal matrices} of the same order. The study and the explicit computation of the fibers of this map are relevant subjects in matrix theory and in its applications (see for instance \cite{Hi2008}). Here we analyse some of their differential-geometric properties. 

The set of all \emph{real skew symmetric logarithms} of $R \in SO_n$ (i.e. the fiber over a matrix $R$) can be described in terms of the set $\mathcal{A}plog(R)$ of its \emph{real skew symmetric principal logarithms} (i.e. the real skew symmetric logarithms of $R$ with eigenvalues having absolute value in $[0, \pi]$) and of linear combinations with integer coefficients of suitable skew symmetric matrices (Theorem \ref{all_skew_symmetric_logarithms}). 

$\mathcal{A}plog(R)$ is (implicitly) studied in \cite{GaXu2002}, where the so-called \emph{Rodrigues exponential formula} for skew symmetric matrices of order $3$ is extended to any order $n$ (Proposition \ref{RGX_exp_form}); we point out the role of the \emph{Singular Value Decomposition} of a skew symmetric matrix (Proposition \ref{SVD} and Definition \ref{SVD-system}).

Also the problems of the existence and of uniqueness of real skew symmetric principal logarithms of $R \in SO_n$ are approached in \cite{GaXu2002}, while many differential-geometric properties of $\mathcal{A}plog(R)$, discussed here, could be new: $\mathcal{A}plog(R)$  has a differential-geometric structure, depending on the presence of $-1$ among the eigenvalues of $R$ (Proposition \ref{class_with_Pfaff}, Theorem \ref{Characterize_plog} and Corollary \ref{Cor_Characterize_plog}). In particular in case of matrices having $-1$ as an eigenvalue, it is diffeomorphic to the manifold of real skew symmetric orthogonal matrices of order equal to the multiplicity of the eigenvalue $-1$ and has two connected components.

We are also able to describe all real skew symmetric logarithms of a matrix $R \in SO_n$ in some particular, but relevant, cases, where they form a discrete lattice of rank $\lfloor n/2 \rfloor$ in $\mathcal{A}_n$  (Theorem \ref{partskewsymmlogTHM} and Remark \ref{REMpartskewsymmlogTHM}).

The exponential map is involved in the description of the \emph{geodesic curves} on the manifold $(O_n, g)$ of real \emph{orthogonal matrices} equipped with the metric induced by the \emph{trace metric} $\overline{g}$ or by the \emph{Frobenius metric} $g$  (Recalls \ref{trace_metric}).

In \cite{DoPe2015} we have studied the trace metric on the whole manifold of real nonsingular matrices $GL_n$, where it defines a structure of semi-Riemannian manifold. This metric is often considered also in the setting of positive definite real matrices (see for instance \cite{Lan1999} Chapt.XII, \cite{BhaH2006} \S2, \cite{Bha2007} Chapt.6, \cite{MoZ2011} \S3), where it defines a structure of Riemannian manifold.
On $O_n$ the trace metric $\overline{g}$ is the opposite of the Frobenius metric $g$ (Lemma \ref{g=-g-segnato}). $(O_n, g)$ is an Einstein Riemannian manifold whose main properties are listed in \ref{conseguenze_bi-invarianza}. Moreover we get a suitable \emph{foliation} on $GL_n$ with leaves isometric to $O_n$ (Proposition \ref{foliations}).

We describe the geodesic curves on $O_n$ with respect to $g$ (and to $\overline{g}$) (Proposition \ref{conseguenze_bi-invarianza} (b), Remark \ref{princ_geod} and Proposition \ref{rat_numb_SO_n}) and in particular the minimal geodesics joining $G$ and $H$, which turn out to be in bijection with the skew symmetric principal logarithms of $G^{-1} H$, furthermore we express the distance $d(G,H)$ in terms of the eigenvalues of $G^{-1} H$ (Theorem \ref{minimalgeodesics} and  Remark \ref{minimal_length}).

After computing the \emph{diameter} of $(O_n, g)$ as $\sqrt{2 \lfloor n/2 \rfloor} \pi$ (Corollary \ref{misura-diametro}), we introduce the notions of \emph{weakly diametral pair of points} and of \emph{diametral pair of points} among the pairs of real orthogonal matrices (Definition \ref{weakly_diametral}) and characterize them in terms of the manifolds of real symmetric orthogonal matrices of order $n$ with the eigenvalues $1$ of multiplicity $p$ and $-1$ of multiplicity $n-p$, each one of them is diffeomorphic to the Grassmannian of $p$-dimensional vector subspaces of $\RR^n$ (Propositions \ref{grassmannian}, \ref{weakly_diametral_pairs} and \ref{diametral_pairs}).

\medskip

\textbf{Acknowledgement.} We want to thank the anonymous referee for many useful and precious suggestions about the matter and the writing of this paper.

\section{Singular Value Decomposition for skew symmetric matrices}

\begin{recalls}\label{initial_recalls}
In this paper all matrices are supposed to be square of order $n$. 

We denote by $M_n$, $\mathcal{A}_n$, $GL_n $, $O_n$ and $SO_n$ respectively the vector space of real matrices of order $n$, its subspace of skew symmetric matrices, the multiplicative group of nondegenerate matrices of $M_n$, the group of real orthogonal matrices and its special subgroup. 

$O_n$ is a differentiable submanifold of $GL_n$ of dimension $n(n-1)/2$ with two connected components $SO_n$ and $O_n^-$ (the orthogonal matrices with determinant $-1$).

As usual $I=I_n$ is the identity matrix of order $n$ and we put
$E_0 := 
\left(
\begin{array}{rr}
0 & 1  \\ 
-1 & 0  
\end{array}
\right)
$. Note that $E_0$ and $-E_0 = E_0^{-1} =E_0^T$ ($A^T$ is the transpose of the matrix $A$) are the unique real skew symmetric orthogonal matrices of order $2$ and that $-E_0 = P_0^T E_0 P_0$, where $P_0 = P_0^T$ is the permutation (orthogonal) matrix 
$ 
\left(
\begin{array}{rr}
0 & 1  \\ 
1 & 0  
\end{array}
\right)
$. Analogously: $-diag(\underbrace{E_0, \cdots , E_0}_{m}) = diag(\underbrace{P_0, \cdots , P_0}_{m})^T \, diag(\underbrace{E_0, \cdots , E_0}_{m}) \, diag(\underbrace{P_0, \cdots , P_0}_{m})$, where $diag(A_1, \cdots , A_m)$ denotes the \emph{block diagonal matrix} with blocks $A_1, \cdots , A_m$.

\medskip

Next Proposition collects some facts about the \emph{Singular Value Decomposition} of a skew symmetric matrix; for general information on this subject we refer for instance to \cite{HoJ2013}  and to \cite{OttPaol2015} for a more geometric point of view.
\end{recalls}

\begin{prop}\label{SVD}
Every matrix $A \in \mathcal{A}_n \setminus \{ 0 \}$ has a unique Singular Value Decomposition

$$
(*) \ \ \ A = \sum_{j = 1}^s \zeta_j A_j
$$

where $\zeta_1, \cdots \zeta_s$ are the distinct nonzero singular values of $A$, while the matrices $A_1, \cdots , A_s$ are  in $\mathcal{A}_n \setminus \{ 0 \}$ and are uniquely determined in this set by $(*)$ and by the conditions:

$A_j^3 = - A_j$ for every $j = 1 , \cdots , s$  and

$A_jA_h= 0$ as soon as $j \ne h$

(when $s=1$, the second condition is vacuous).
\end{prop}
 
\begin{proof}
It is well-known that every $A \in \mathcal{M}_n \setminus \{ 0 \}$ admits a unique Singular Value Decomposition
$
A = \sum_{j = 1}^s \zeta_j A_j  \  \  ,
$
where $\zeta_1, \cdots \zeta_s$ are the distinct nonzero singular values of $A$, and the matrices $A_1, \cdots , A_s \in \mathcal{M}_n \setminus \{ 0 \}$ are uniquely determined in this set by $(*)$ and by the conditions:

$A_j A_j^T A_j=A_j$ for every $j = 1 , \cdots ,s$ and

$A_k A_j^T=A_j^T A_k=0$ as soon as $j \neq k$.

We refer for instance to \cite{OttPaol2015} \ Theorem 3.4, for a proof of this statement. By the skew symmetry of $A$ and the uniqueness of the Singular Value Decomposition we get easily the skew symmetry of every $A_j$'s.  So the previous conditions about the $A_j$'s
become equivalent to the conditions of our statement.  
\end{proof}

\begin{defi}\label{SVD-system}
From now on, we will say that a set of nonzero skew symmetric matrices $A_1, \cdots , A_s$ is an \emph{SVD system} if $A_j^3 = - A_j$ for every $j = 1 , \cdots , s$  and $A_jA_h= 0$ as soon as $j \ne h$. 

The previous Proposition justifies this definition.

\medskip

The exponential map of a linear combination of an SVD system has an interesting expression, we refer as Rodrigues exponential formula:
\end{defi}

\begin{prop}[see \cite{GaXu2002} \S 2]\label{RGX_exp_form} $ \\ $ Let $A_1, \cdots , A_s \in \mathcal{A}_n \setminus \{ 0 \}$ be an SVD system. Then 
$$
exp(\sum_{j =1}^s \zeta_jA_j) = I_n + \sum_{j = 1}^s[sin(\zeta_j) A_j + (1-\cos(\zeta_j))A_j^2].
$$
\end{prop}

\begin{rem}

Note that 

$$I_n + \sum_{j = 1}^s(1-\cos(\zeta_j))A_j^2 \mbox{\ \ \ \ and \ \ \ \  } \sum_{j = 1}^s \sin(\zeta_j) A_j$$

are respectively the symmetric part and the skew symmetric part of the exponential of the skew symmetric matrix $\sum_{j = 1}^s \zeta_j A_j$.
\end{rem}

\begin{rem}\label{eigenvalues_B}
From standard facts on singular values we get that, if $A = \sum_{j =1}^s \zeta_j A_j$ is a skew symmetric matrix with its Singular Value Decomposition, then
$A, A_1, \cdots , A_s$ are simultaneously diagonalizable over $\CC$, the eigenvalues of 
$A$ are $\pm \mathbf{i} \zeta_k$ both with the same multiplicity $m_k = rank(A_k)/2$, for every k = 1,...,s,  and possibly $0$ with multiplicity $n - 2\sum_{k=1}^s m_k$.

Moreover  $rank(A) = 2 \sum_{k=1}^s m_k$ and $tr(A^2) = - 2 \sum_{k=1}^s m_k \zeta_k^2 = - \sum_{l=1}^n  |\lambda_l|^2$, where $\lambda_1 , \cdots , \lambda_n$ are all the (possibly repeated) eigenvalues of $A$.
\end{rem}

\section{Real skew symmetric logarithms of special orthogonal matrices} 

\begin{notas}\label{eigenvalues_R}
In this and in the following section $R$ is a fixed matrix in $SO_n$. Since the eigenvalues of $R$ have absolute value $1$, we can  express them in this way: $e^{\pm \mathbf{i} \theta_1}$,  both with multiplicity $m_1$, $e^{\pm \mathbf{i} \theta_2}$, both with multiplicity $m_2$, until $e^{\pm \mathbf{i} \theta_p}$, both with multiplicity $m_p$, where $\theta_1, \cdots , \theta_p \in (0, \pi]$ are distinct and, in addition, possibly $1=e^0$, with multiplicity $n-2m$, where $m= m_1 + \cdots + m_p$, with the convention that, if $-1 = e^{\pm \mathbf{i} \pi}$ is an eigenvalue of $R$, then $\theta_1 = \pi$ and in this case $-1$ has multiplicity $2m_1$.
\end{notas}

\begin{defi} 
A \emph{real skew symmetric logarithm} of $R$ is any real skew symmetric matrix $B$ such that $exp(B) = R$.

A \emph{real skew symmetric principal logarithm} of $R$ is any real skew symmetric logarithm of $R$  such that its eigenvalues have absolute value in $[0, \pi]$.

We denote by $\mathcal{A}plog(R)$ the set of real skew symmetric principal logarithms of $R$. In particular $\mathcal{A}plog(I_n)$  consists only of the null matrix.
\end{defi}

\begin{prop}\label{inequality}
Assume the Notations \ref{eigenvalues_R} and suppose $R \ne I_n$.

1) $A \in {M}_n$ is a skew symmetric principal logarithm of $R \in SO_n$ if and only if 
$$A= \sum_{k=1}^p \theta_k B_k$$ 
is its Singular Value Decomposition and $exp(\sum_{k=1}^p \theta_k B_k) =R$.

2) For every skew symmetric principal logarithm $B$ of $R \in SO_n$ we have 
$$tr(B^2) = - 2 \sum_{k=1}^p m_k \theta_k^2 \  \ ,$$ 
so $tr(B^2)$ is constant on $\mathcal{A}plog(R)$.

3) If $A$ is a real skew symmetric logarithm of $R \in SO_n$, then 
$$tr(A^2) \leq - 2 \sum_{k=1}^p m_k \theta_k^2$$ 
with equality if and only if $A$ is a skew symmetric principal logarithm of $R$.
\end{prop}

\begin{proof}
1) If $exp(A) = R$, the eigenvalues of $A$ are complex logarithms of the eigenvalues of $R$. Therefore $A \in \mathcal{A}plog(R)$ if and only if it is a real skew symmetric logarithm of $R$ and its eigenvalues are $\pm \mathbf{i} \theta_1$, both with multiplicity $m_1$,  $\pm \mathbf{i} \theta_2$, both with multiplicity $m_2$, until $\pm \mathbf{i} \theta_p$, both with multiplicity $m_p$, and $0$ with multiplicity $n-2m$.
This fact together with Proposition \ref{SVD} gives the statement.

2) It follows directly from  \ref{eigenvalues_B}.

3) Assume $exp(A)= R$ with $A \in \mathcal{A}_n$. We pose

$\theta_1' = \cdots = \theta_{m_1}' = \theta_1$, next  $\theta_{m_1 +1}' = \cdots = \theta_{m_1+ m_2}' = \theta_2$, until  

$\theta_{m_1 + \cdots + m_{p-1}+1}' = \cdots = \theta_m' = \theta_p$.

The eigenvalues of $A$ are necessarily of the form 

$$\pm \varphi_1 \mathbf{i}, \cdots , \pm \varphi_m \mathbf{i}, \ \ \pm 2 \pi h_1 \mathbf{i}, \cdots , \pm 2 \pi h_s \mathbf{i}, \ \  \underbrace{0, \cdots , 0}_{n-2(s+m)},$$

where  $s \le (n-2m)/2,\  \ $    $h_r \in \ZZ$ for every $r= 1, \cdots , s$.

We can assume that $\varphi_j = \overline{\theta}_j + 2 \pi k_j>0$ with $ k_j \in \ZZ$ and $\overline{\theta}_j = \pm\theta'_j$ for every $j= 1, \cdots , m$ and, if $\theta_1 = \pi$, then $\overline{\theta}_j = \pi$ for every $j= 1, \cdots , m_1$.

Standard computations show that $k_j \ge 0 $ and $\pi k_j + \overline{\theta}_j >0$  for every $j=1, \cdots , m$.

Hence, for every $j$, we have $k_j(\pi k_j + \overline{\theta}_j) \ge 0$ with equality if and only if $k_j =0$.

By \ref{eigenvalues_B} we have 

\begin{multline*}
tr(A^2) = -2 \sum_{j=1}^m (\overline{\theta}_j + 2 \pi k_j)^2 - 8 \pi^2 \sum_{j=1}^s h_j^2 =\\ 
-2 \sum_{j=1}^m \overline{\theta}_j^2 -8\pi(\sum_{j=1}^m k_j(\pi k_j + \overline{\theta}_j) + \pi \sum_{j=1}^s h_j^2) \le 
-2 \sum_{j=1}^m \overline{\theta}_j^2 = -2 \sum_{j=1}^m \theta_j'^2 = tr(B^2) ,
\end{multline*}

where $B \in \mathcal{A}plog(R)$ and the equality holds if and only if $k_j =0$ for every $j$ and $h_r=0$ for every $r$, i.e. if and only if $A \in \mathcal{A}plog(R)$.
\end{proof}

\begin{thm}\label{all_skew_symmetric_logarithms}
A real skew symmetric matrix $A$ of order $n$ is a real skew symmetric logarithm of $R \in SO_n$ if and only if there exist a skew symmetric principal logarithm $B$ of $R$, an SVD system of skew symmetric matrices $C_1, \cdots , C_s$ all commuting with $B$ , and $l_1, \cdots , l_s \in \ZZ$ such that
$$
A = B + 2\pi \sum_{j= 1}^s l_j C_j.
$$
\end{thm}

\begin{proof}
If $A = B + 2 \pi \sum_{j=1}^s l_j C_j$ as above, then by Rodrigues exponential formula

$exp(A) = exp(B) \, exp( 2 \pi \sum_{j=1}^s l_j C_j) = exp(B) = R$.

For the converse, let $A$ be any real skew symmetric logarithm of $R$ and let 

$\pm \mathbf{i} \zeta_1, \cdots, \pm \mathbf{i} \zeta_s $, $2s \le n$, $\zeta_j >0$, be the distinct nonzero eigenvalues of $A$. By Theorem \ref{SVD}, $A$ has a unique Singular Value Decomposition: $A = \sum_{j=1}^s \zeta_j A_j$. 

For every $j$, there is a unique pair $(\eta_j, k_j)$ with $\eta_j \in (0, \pi]$ and $k_j \in \NN$ such that $\zeta_j = \eta_j + k_j \pi$, so, for every $j$, we pose

\smallskip

$
(z_j, \tau_j, l_j, C_j) =
\begin{cases}
(\zeta_j, \eta_j, k_j /2 , A_j) \mbox{ if } k_j \mbox{ is even}\\
(-\zeta_j, \pi - \eta_j, (-k_j -1) / 2, -A_j) \mbox{ if } k_j \mbox{ is odd}.
\end{cases}
$

\smallskip

Hence the nonzero distinct eigenvalues of $A$ are $\pm \mathbf{i} z_1, \cdots , \pm \mathbf{i} z_s$ with $z_j = \tau_j + 2 l_j \pi$ where $\tau_j \in [0, \pi]$ and $l_j \in \ZZ$ and we get  \   \ $A =  \sum_{j=1}^s z_j C_j$.

The eigenvalues of $R$ are the exponentials of the eigenvalues of $A$, i.e. $e^{\pm \mathbf{i} \tau_1}, \cdots , e^{\pm \mathbf{i} \tau_s}$ with $\tau_1, \cdots , \tau_s \in [0, \pi]$ and possibly $1= e^0$. 

Remembering the Notations \ref{eigenvalues_R}, up to reorder the indices, we can assume that the first $m_1$, among the $\tau_j$'s, are equal to $\theta_1$, the next $m_2$ are equal to $\theta_2$, until the last  $m_p$ nonzero equal to $\theta_p$ and that the remaining $s-(m_1 + \cdots + m_p) = s-m$ are zero.

We set $B_1 = C_1 + \cdots + C_{m_1}$, $B_2 = C_{m_1 + 1} + \cdots + C_{m_1 + m_2}$, until \\$B_p = C_{m_1+ \cdots + m_{p-1} + 1} + \cdots + C_m
$. 
Then \\ 
$
A = \sum_{j =1}^m \tau_j C_j + 2 \pi \sum_{j= 1}^s l_j C_j = \sum_{j = 1}^p \theta_j B_j + 2 \pi \sum_{j= 1}^s l_j C_j
$.

To conclude, we check that $B=\sum_{j = 1}^p \theta_j B_j$ is a skew symmetric principal logarithm of $R$. 

We remark that the $B_j$'s form an SVD system, indeed so are the $C_j$'s. Furthermore $B$ commutes with any $C_j$, hence $R= exp(A) = exp(B)  exp(2 \pi \sum_{j=1}^s l_j C_j) = exp(B)$, by Rodrigues exponential formula and so $B$ is a real skew symmetric principal logarithm of $R$, by Proposition \ref{inequality} (1).
\end{proof}

\begin{cor}\label{skew_symmetric_and_principal_logarithm}
If $R \in SO_n$, then $\mathcal{A}plog(R) \ne \emptyset$.
\end{cor}

\begin{proof}
It follows from the previous Theorem \ref{all_skew_symmetric_logarithms} via the existence of real skew symmetric logarithms of $R$ (see for instance \cite{BrtD1985} IV, 2.2).
\end{proof}

\begin{rem}\label{explicitlogprin}
The existence of a real skew symmetric principal logarithm of $R$ is constructively obtained in \cite{GaXu2002} Lemma 2.4. We resume here that construction.

First of all we observe that, for every $\theta \in \RR$, $exp(\theta E_0) = 
\left(
\begin{array}{rr}
\cos \theta & \sin \theta  \\ 
-\sin \theta & \cos \theta  
\end{array}
\right)
$.

As in \cite{HoJ2013} Cor.2.5.11 (c), there is a real orthogonal matrix $K$ such that

$R= K diag(\underbrace{exp(\theta_1 E_0), \cdots , exp(\theta_1 E_0)}_{m_1}, \cdots , \underbrace{exp(\theta_p E_0), \cdots , exp(\theta_p E_0)}_{m_p}, \underbrace{1, \cdots ,1}_{n-2m} )K^T$

\medskip
$=exp(K diag(\underbrace{\theta_1 E_0, \cdots , \theta_1 E_0}_{m_1}, \cdots , \underbrace{\theta_p E_0, \cdots , \theta_p E_0}_{m_p}, \underbrace{0, \cdots ,0}_{n-2m} )K^T )$

\medskip

$=exp(\sum_{j=1}^p \theta_j B_j),$\      \ where $B_1= K diag(\underbrace{E_0, \cdots , E_0}_{m_1}, \underbrace{0, \cdots, 0}_{n-2m_1})K^T,$\\

\medskip
$B_2=K diag(\underbrace{0,\cdots,0}_{2m_1},\underbrace{E_0, \cdots , E_0}_{m_2},\underbrace{0, \cdots, 0}_{n-2(m_1+m_2)})K^T$,

\medskip
until $B_p = K diag(\underbrace{0, \cdots , 0,}_{2(m_1 + \cdots + m_{p-1})} \underbrace{E_0, \cdots , E_0}_{m_p}, \underbrace{0, \cdots , 0}_{n-2m})K^T$.

\medskip

Then $B= \sum_{j=1}^p \theta_j B_j$ is the real skew symmetric principal logarithm constructed in \cite{GaXu2002} with its Singular Value Decomposition.

Also its uniqueness, exactly when $-1$ is not an eigenvalue of $R$, is proved in \cite{GaXu2002} Thm. 4.1. Precisely the following result holds
\end{rem}

\begin{prop}\label{system_R}
Let $B= \theta_1 B_1 + \cdots + \theta_p B_p$ with $\theta_1, \cdots , \theta_p \in (0, \pi]$ be any real skew symmetric principal logarithm of $R \in SO_n$ together with its Singular Value Decomposition. 

1) If $-1$ is not an eigenvalue of $R$, then $B_1, B_2, \cdots, B_p$ are uniquely determined  by $R$, so there is a unique real skew symmetric principal logarithm of $R$.

2) If $-1$ is an eigenvalue of $R$ (i.e. if $\theta_1=\pi$), then $B_2, \cdots, B_p$ are uniquely determined by $R$ and moreover also $B_1^2$ is uniquely determined by $R$ because we have

$$B_1^2 = \frac{1}{4} (R+R^T) -\frac{1}{2} I_n -\frac{1}{2} \sum_{j=2}^p (1-\cos\theta_j)B_j^2.$$ 
\end{prop}

\begin{proof}
See \cite{GaXu2002} Thm. 4.1 and its proof.
\end{proof}

\section{Differential-geometric properties of $\mathcal{A}plog(R)$}

\begin{rem}\label{OSSMatrices} 
In order to characterize $\mathcal{A}plog(R)$ with $R \in SO_n$, we must study $\mathcal{M}_{2\mu}$: the  compact subset of $\mathcal{A}_{2\mu}$ of skew symmetric orthogonal matrices of order $2 \mu$.

The Lie group $O_{2\mu}$ acts on $\mathcal{A}_{2\mu}$ on the left as $\phi(Q, M) = Q M Q^T$ and $\mathcal{M}_{2\mu}$ is the orbit of the matrix $diag(\underbrace{E_0, \cdots , E_0}_{\mu})$. The isotropy group of $diag(\underbrace{E_0, \cdots , E_0}_{\mu})$ is $Sp_{2\mu} \cap O_{2\mu}$, where $Sp_{2\mu} = Sp_{2\mu}(\RR)$ is the symplectic group of real matrices of order $2\mu$. It is a standard fact that this intersection is isomorphic to the unitary group $U_\mu$ of complex matrices of order $\mu$, then  $\mathcal{M}_{2\mu}$ is a compact submanifold of $\mathcal{A}_{2\mu}$ diffeomorphic to $O_{2\mu} / U_\mu$ (see for instance \cite{Popov1991} p.1). Its dimension is equal to $dim(O_{2\mu})-dim(U_{\mu})={\mu}({\mu}-1).$
\end{rem}

\begin{prop}\label{class_with_Pfaff}
With the same notations as in Remark \ref{OSSMatrices}, the manifold $\mathcal{M}_{2\mu}$  is compact of dimension $\mu(\mu-1)$ and its connected components are $\mathcal{M}_{2\mu}^+$ and $\mathcal{M}_{2\mu}^-$: the sets of real skew symmetric orthogonal matrices of order $2\mu$ with Pfaffian equal to $1$ and $-1$ respectively. 
These components are simply connected manifolds both diffeomorphic to the homogeneous space $SO_{2 \mu}/ U_\mu$.
In particular if $\mu=1$, each of them is a single point and if $\mu=2$, they are both diffeomorphic to a $2$-sphere.
\end{prop}

\begin{proof}
We refer for instance to \cite{Onish1991} for the \emph{Pfaffian}, $\mathcal{P}f(M)$, of a skew symmetric matrix $M$.

Every real skew symmetric orthogonal matrix of order $2\mu$  can be written as \\
$Q \, diag(\underbrace{E_0, \cdots , E_0}_{\mu}) \, Q^T$ with $Q \in O_{2\mu}$.

Now $\mathcal{P}f(diag(\underbrace{E_0, \cdots , E_0}_{\mu}))=1$, so
  
$\mathcal{P}f(Q \, diag(\underbrace{E_0, \cdots , E_0}_{\mu}) \, Q^T) = det(Q) \, \mathcal{P}f(diag(\underbrace{E_0, \cdots , E_0}_{\mu})) = det(Q) = \pm 1$.

Hence the map $(Q,E) \mapsto Q E Q^T$ defines a transitive action of $SO_{2\mu}$ on $\mathcal{M}_{2\mu}^+$ (which is diffeomorphic to $\mathcal{M}_{2\mu}^-$ by means of a congruence with an orthogonal matrix of determinant $-1$). This implies that $\mathcal{M}_{2\mu}^+ \simeq SO_{2\mu}/U_\mu$, hence $\mathcal{M}_{2\mu}^+$ is connected, being $SO_{2\mu}$ connected. From the continuity of the Pfaffian map, $\mathcal{M}_{2\mu}^+$ and $\mathcal{M}_{2\mu}^-$ are the connected components of $\mathcal{M}_{2\mu}$.

The case $\mu = 1$ is trivial and already noted in Recalls \ref{initial_recalls}.

It is known that Bott has computed the homotopy groups of order $k \le 2\mu -2$ of such quotients (see \cite{Bott1959} and also \cite{Massey1961}). In particular the fundamental group is trivial and so $\mathcal{M}_{2\mu}^+$ is simply connected. Moreover, if $\mu =2$, $\mathcal{M}_{4}^+$ is homeomorphic (therefore diffeomorphic) to a $2$-sphere; this case can be also deduced as a consequence of \cite{Pearl1959} Thm.1.
\end{proof}

\begin{thm}\label{Characterize_plog}
Assume the same notations and conventions introduced in \ref{eigenvalues_R}.

1) If $-1$ is not an eigenvalue of $R \in SO_n$, then $R$ has a unique skew symmetric principal logarithm.

2) If $-1$ is an eigenvalue of $R \in SO_n$ (i.e. if $\theta_1=\pi$)  with multiplicity $2m_1 \geq 2$ and $B= \sum_{j=1}^p \theta_j B_j$ is any real skew symmetric principal logarithm of $R$ together with its Singular Value Decomposition, then $\mathcal{A}plog(R)$ is the embedded submanifold of $\mathcal{A}_n$ diffeomorphic to $\mathcal{M}_{2m_1}$ defined by

\begin{multline*} 
\mathcal{A}plog(R) =
 \{  \pi W + \theta_2 B_2 + \cdots + \theta_p B_p \ /  \\
 W \in \mathcal{A}_n ; W, B_2, \cdots , B_p \mbox{\  form an SVD system} \mbox{\ and \ } W^2 = B_1^2 \}.
\end{multline*}
\end{thm}

\begin{proof}
Part (1) is Proposition \ref{system_R} (1).

When $-1$ is an eigenvalue of $R$, by Proposition \ref{system_R} (2), $B_1^2, B_2, \cdots , B_p$ are uniquely determined by $R$.

From this observation and from Proposition \ref{inequality} we get that every real skew symmetric principal logarithm of $R$ has the form 
$\pi W + \theta_2 B_2 + \cdots + \theta_p B_p$ 
where $W \in \mathcal{A}_n$ and $W, B_2, \cdots , B_p$ is an SVD system with $W^2 = B_1^2$.

We remark that if a matrix has the form $\pi W + \theta_2 B_2 + \cdots + \theta_p B_p$, with $W$ as above, then it is a real skew symmetric principal logarithm of $R$. 

Indeed 
\begin{multline*}
exp(\pi W + \theta_2 B_2 + \cdots + \theta_p B_p) = I_n + \sum_{j=2}^p[\sin\theta_j B_j + (1-\cos\theta_j)B_j^2] + 2 W^2 =\\
I_n + \sum_{j=2}^p[\sin\theta_j B_j + (1-\cos\theta_j)B_j^2] + 2 B_1^2 = R.
\end{multline*}

Now let $W$ be a skew symmetric matrix with $W^2 = B_1^2$ and such that  $W, B_2, \cdots , B_p$ form an SVD system.

As in \cite{GaXu2002} \S 2 (and as in Remark \ref{explicitlogprin}) we can write
$
W^2 =
B_1^2 = [K \ 
diag( \underbrace{
E_0,
\cdots ,
E_0
}_{m_1},
\underbrace{0, \cdots , 0}_{n-2m_1}
)
\ K^T]^2 =
-  K \ diag(I_{2m_1}, \underbrace{0, \cdots , 0}_{n-2m_1}) \ K^T 
$ with $K \in O_n$, so $W^2$ has $-1$ as eigenvalue of multiplicity $2 m_1$ and $0$ as eigenvalue of multiplicity $n - 2m_1$. Hence the eigenvalues of the real matrix $W$ must be $\pm \mathbf{i}$ both with multiplicity $m_1$ and $0$ with multiplicity $n-2m_1$. 

Since $W$ is skew symmetric, there exists a real orthogonal matrix $Q$ such that
$$W = Q \, diag( \underbrace{
E_0,
\cdots ,
E_0
}_{m_1},
\underbrace{0, \cdots , 0}_{n-2m_1}) \, Q^T$$
(see \cite{HoJ2013} Cor. 2.5.11 (b)).

From $W^2 = B_1^2$, we get $K^T Q \ diag(I_{2m_1}, \underbrace{0, \cdots , 0}_{n-2m_1}) = diag(I_{2m_1}, \underbrace{0, \cdots , 0}_{n-2m_1}) \ K^TQ $.

Therefore $K^TQ$ is an orthogonal matrix which commutes with the block diagonal matrix $diag(I_{2m_1}, \underbrace{0, \cdots , 0}_{n-2m_1})$. This is equivalent to say that
$K^TQ$ is block diagonal too, say: $diag(H_{2m_1}, H_{n - 2m_1})$. But $K^TQ$ is orthogonal, so the blocks must be orthogonal matrices too of orders $2m_1$ and $n-2m_1$ respectively. Hence
\begin{multline*}W = K \ diag(H_{2m_1}, H_{n - 2 m_1}) \, diag( \underbrace{
E_0,
\cdots ,
E_0
}_{m_1},
\underbrace{0, \cdots , 0}_{n-2m_1}) \,   \ diag(H_{2m_1}^T, H_{n - 2m_1}^T) \ K^T \\
=K \ diag(H_{2m_1} \ 
diag(
\underbrace{
E_0,
\cdots ,
E_0
}_{m_1}) 
\ H_{2m_1}^T, 
\underbrace{0, \cdots , 0}_{n-2m_1}) \
K^T.
\end{multline*}

Now (again by \cite{HoJ2013} Cor. 2.5.11 (b)) 
$
H_{2m_1} \ 
diag(
\underbrace{
E_0,
\cdots ,
E_0
}_{m_1}) 
\ H_{2m_1}^T
$ is the form of a generic real skew symmetric orthogonal matrix of order $2m_1$, hence $W$ can be written as $K   
diag(M, \underbrace{0, \cdots , 0}_{n-2m_1}) 
K^T
$,
where $M$ is a skew symmetric orthogonal matrix of order $2m_1$.

Vice versa if
$W =
K   
diag(M, \underbrace{0, \cdots , 0}_{n-2m_1}) 
K^T
$, 
with $M$ as before, it is easy to check that $W^3 = -W$, $WB_j = B_j W =0$ for any $j = 2, \cdots , p$, $W^2 = B_1^2$ and $W \in \mathcal{A}_n$. 

Moreover 

$R= I_n + \sum_{j=2}^p[\sin\theta_jB_j + (1-\cos\theta_j)B_j^2)] + 2W^2 =
exp(\pi W + \theta_2 B_2 + \cdots + \theta_p B_p)
$. 

This implies that $\psi: \mathcal{M}_{2m_1} \to \mathcal{A}plog(R) \subset \mathcal{A}_n$, defined by 

$$\psi(M) = \pi K diag(M, \underbrace{0, \cdots , 0}_{n-2m_1})K^T + \sum_{h=2}^p \theta_h B_h$$

is a bijection. We can conclude, because $\psi$ is clearly a $C^\infty$-embedding of $\mathcal{M}_{2m_1}$ into $\mathcal{A}_n$.
\end{proof}

\begin{cor}\label{Cor_Characterize_plog}
Assume that $-1$ has multiplicity $2\mu$ as eigenvalue of $R \in SO_n$ ($\mu=0$ means that $-1$ is not eigenvalue).

If $\mu=0$, $\mathcal{A}plog(R)$ consists of a single point.

If $\mu=1$, $\mathcal{A}plog(R)$ consists of two distinct points.

If $\mu=2$, $\mathcal{A}plog(R)$ is diffeomorphic to the disjoint union of two $2$-spheres.

If $\mu \geq 3$, $\mathcal{A}plog(R)$ has two mutually diffeomorphic connected components each of them has the following homotopy groups of order $k \le 2\mu -2$:

$\pi_k \simeq 
\begin{cases}
0 \mbox{ for } k \equiv 1,3,4,5 \mbox{\ \ \ \  (mod 8)}\\
\ZZ \mbox{ for } k \equiv 2,6 \mbox{\ \ \ \ \ \ \ \ \  (mod 8)}\\
\ZZ_2 \mbox{ for } k \equiv 0,7 \mbox{\ \ \ \ \ \ \ \  (mod 8).}
\end{cases}
$

\end{cor}

\begin{proof}
It follows from Theorem \ref{Characterize_plog} and from Proposition \ref{class_with_Pfaff}, while the above homotopy groups of $SO_{2\mu}/U_\mu$, said to be \emph{stable}, are computed in \cite{Bott1959} (see also \cite{Massey1961}).
\end{proof}

\begin{rem}
In the statement of the previous Corollary we have mentioned only the stable homotopy groups of $SO_{2\mu}/U_\mu$. Some other results, about \emph{unstable} homotopy groups (i.e. of order $k \ge 2\mu-1$), could be stated, by considering further studies since \cite{Massey1961} and \cite{Harris1963}.
\end{rem}

\begin{rem}\label{partskewsymmlogREM}
We conclude this section by giving an explicit representation of the set of all real skew symmetric logarithms in some particular, but relevant, cases.

From now on and until the end of this Section, $R$ is a fixed matrix in $SO_n$ such that all the following conditions hold:

- its nonreal complex eigenvalues are all of multiplicity $1$;

- either $-1$ is not an eigenvalue or it has multiplicity $2$;

- either $1$ is not an eigenvalue or it has multiplicity at most $2$.

These are equivalent to  $m=p$ and $n-2p \le 2$ (remember Notations \ref{eigenvalues_R}). Note that $m=p$ is equivalent to $m_1= \cdots = m_p = 1$.

Let $B = \sum_{j=1}^p \theta_j B_j$ be the real skew symmetric principal logarithm of $R$, with its Singular Value Decomposition, whose construction is recalled in Remark \ref{explicitlogprin}. Remember that  there exists a matrix $K \in O_n$ such that $B_j = Kdiag(\underbrace{0, \cdots , 0}_{2j-2}, E_0, \underbrace{0, \cdots , 0}_{n-2j})K^T$, $1 \le j \le p$.

When $1$ is an eigenvalue of $R$ with multiplicity $2$ (i.e. if $n=2p+2$), we pose $\theta_{p+1}=0$, $B_{p+1} = Kdiag(\underbrace{0, \cdots , 0}_{2p}, E_0)K^T$. 

Hence, if $q= \lfloor n/2 \rfloor$ (the integer part of $n/2$), in any case we can write \\ $B = \sum_{j=1}^q \theta_j B_j$ \  \ .

If $\theta_1 \ne \pi$ the previous one is the unique real skew symmetric principal logarithm of $R$, while  if $\theta_1 = \pi$, the only two real skew symmetric principal logarithms of $R$ are 
$\pi B_1 + \sum_{j=2}^q \theta_j B_j$ and $-\pi B_1 + \sum_{j=2}^q \theta_j B_j$ (remember Theorem \ref{Characterize_plog} and Corollary \ref{Cor_Characterize_plog}).

Finally for $F_j = diag(\underbrace{0, \cdots , 0}_{2j-2}, E_0, \underbrace{0, \cdots , 0}_{n-2j})$, we have $B_j = K F_j K^T$ for $1 \le j \le q$.
\end{rem}

\begin{lemma}\label{partskewsymmlogLEMMA}
With the same notations as in Remark \ref{partskewsymmlogREM}, assume $q =\lfloor n/2 \rfloor$ and let $B_0 \in \mathcal{A}_{n}$ defined by $B_0 = \sum_{j=1}^q \psi_j F_j$, where the $\psi_j$'s are real numbers such that $\psi_j^2 \ne \psi_h^2$ as soon as $j \ne h$ when $n$ is even, and  $\psi_j^2 \ne \psi_h^2$ as soon as $j \ne h$ and $\psi_j \ne 0$ for every $j$ when $n$ is odd.

Then $C_0 \in \mathcal{A}_n$ commutes with $B_0$ if and only if $C_0 = \sum_{j=1}^q \beta_j F_j$ for some $\beta_1, \cdots , \beta_q \in \RR$.
\end{lemma}

\begin{proof}
One implication is trivial.

Assume first that $n = 2q$. We write the matrix $C_0$ as a block matrix with $2 \times 2$ blocks:
$C_0 = (\Lambda_{ij})_{1 \le i,j \le q}$. Then the condition $B_0C_0 = C_0 B_0$ is equivalent to
$$
\psi_j E_0 \Lambda_{jh} = \psi_h \Lambda_{jh} E_0 \mbox{\ \ for every\ \ } j, h = 1, \cdots q. 
$$

Fixed $j,h$, we denote $\Lambda_{jh}= 
\left(
\begin{array}{rr}
a & b  \\ 
c & d  
\end{array}
\right)
$, so the previous conditions give 

$$
\psi_j c= -\psi_h b, \ \ -\psi_j b = \psi_h c ,\ \ \psi_j a = \psi_h d, \ \ \psi_j d = \psi_h a. 
$$

and therefore

$$
(\psi_h^2 -\psi_j^2)ad=0, \ \ (\psi_h^2 - \psi_j^2) bc =0.
$$

Hence, for $j \ne h$, we have $ad=bc=0$ and so easy computations give $\Lambda_{jh} =0$ as soon as $j \ne h$.  Therefore $C_0 = diag(\Lambda_{11}, \cdots , \Lambda_{qq})$. But $C_0$ is skew symmetric, 

so $\Lambda_{jj} = \beta_j E_0$ and we can conclude for $n$ even.

When $n= 2q+1$ we pose

$$
C_0 =
\left(
\begin{array}{cccc}
\Lambda_{11} & \cdots & \Lambda_{1q} & \Lambda_{1, q+1} \\
\Lambda_{21} & \cdots & \Lambda_{2q} & \Lambda_{2, q+1} \\ 
\vdots  & \vdots & \vdots & \vdots  \\
\Lambda_{q1} & \cdots & \Lambda_{qq} & \Lambda_{q, q+1}  \\
\Lambda_{q+1, 1}  & \cdots & \Lambda_{q+1, q} & \lambda  \\
\end{array}
\right)
$$

where, for every $j, h = 1, \cdots , q$, \ $\Lambda_{j, h} \in M_2$, while $\Lambda_{j,q+1}$ and $\Lambda_{q+1,h}$ are respectively column and row $2$-vectors and $\lambda \in \RR$.

From $B_0 C_0 = C_0 B_0$, we get $\Lambda_{jh} =0$, for every $1 \le j \ne h \le q$, as above, and moreover, by looking at the entries $(j, q+1)$ and $(q+1, h)$, we get also that $\Lambda_{j, q+1}$ and $\Lambda_{q+1, h}$ are both zero for every $j, h =1, \cdots , q$, taking into account that all $\psi_h$'s are nonzero. 
Since $C_0$ is skew symmetric, we have $\Lambda_{jj} = \beta_{j} E_0$, for every $j$ and $\lambda =0$. This concludes the proof.
\end{proof}

\begin{thm}\label{partskewsymmlogTHM}
With the same notations as in \ref{eigenvalues_R} and in \ref{partskewsymmlogREM}, let $R \in SO_n$ and suppose that all the following conditions on $R$ hold:

- its nonreal complex eigenvalues are all of multiplicity $1$;

- either $-1$ is not an eigenvalue or it has multiplicity $2$;

- either $1$ is not an eigenvalue or it has multiplicity at most $2$.

Fix $B=\sum_{h=1}^{\lfloor n/2 \rfloor}  {\theta_h} B_h$ $\in$ $\mathcal{A}plog(R)$ (as in \ref{partskewsymmlogREM}).

If $-1$ is not an eigenvalue of $R$ (i.e. $\theta_1 \ne \pi$), then $\mathcal{A}plog(R) = \{ B \}$, while if $-1$ is an eigenvalue of $R$ (i.e. $\theta_1 = \pi$), then $\mathcal{A}plog(R)$ consists of $B = \pi B_1 + \sum_{h=2}^{\lfloor n/2 \rfloor}  {\theta_h} B_h$ and $- \pi B_1 + \sum_{h=2}^{\lfloor n/2 \rfloor}  {\theta_h} B_h$.

Moreover $A \in \mathcal{A}_n$ is a real skew symmetric logarithm of $R$ if and only if 

$$A= \sum_{h=1}^{\lfloor n/2 \rfloor} (\theta_h+ 2 \pi r_h) B_h = B + 2 \pi \sum_{h=1}^{\lfloor n/2 \rfloor}  r_h B_h$$
for some $r_1, \cdots , r_{\lfloor n/2 \rfloor} \in \ZZ$.

 In particular the set of all real skew symmetric logarithms of $R$  is a discrete lattice of rank $\lfloor n/2 \rfloor$ in $\mathcal{A}_n$.
\end{thm}

\begin{proof}

The first sentence follows directly from Remark \ref{partskewsymmlogREM}.

One implication of the equivalence is trivial by Rodrigues exponential formula.

For the other implication assume that $A$ is a real skew symmetric logarithm of $R$. Then, by Theorem \ref{all_skew_symmetric_logarithms}, \  \ $A= \overline{B} + \sum_{j=1}^s 2 \pi l_j C_j$ where

 $ \overline{B} \in \mathcal
{A}plog(R)$,\  \ $C_1, \cdots C_s$ form an SVD system of skew symmetric matrices all commuting with $\overline{B}$ and $l_1, \cdots , l_s \in \ZZ$.

According to Remark \ref{partskewsymmlogREM}, we can express the matrix $\overline{B}$ as $\overline{B}= \sum_{h=1}^q \overline{\theta}_h B_h$, where $q= \lfloor n/2 \rfloor$, $\overline{\theta}_h =\theta_h$ for every $h \geq 2$ and $\overline{\theta}_1 = \theta_1$ or $\overline{\theta}_1 = -\pi$.

Standard computations show that, fixed $j = 1, \cdots , s$, the condition $C_j \overline{B}= \overline{B} C_j$ is equivalent to $K^T C_j K ( \sum_{h = 1}^q \overline{\theta}_h F_h) = ( \sum_{h = 1}^q \overline{\theta}_h F_h) K^T C_j K$. Hence, by Lemma \ref{partskewsymmlogLEMMA}, we get $K^T C_j K = \sum_{h=1}^q \beta_{j,h} F_h$, so $C_j= \sum_{h=1}^q \beta_{j,h} B_h$, where $\beta_{j,h}\in \RR$. Since $C_j^3 = -C_j$, we get $\beta_{j,h} \in\{ -1,0,1 \}$ for every $j,h$.

Hence we have $A = \sum_{h=1}^q \overline{\theta}_h B_h + \sum_{h=1}^q (\sum_{j=1}^s 2 \pi l_j \beta_{j,h}) B_h = \sum_{h=1}^q (\overline{\theta}_h +2 \pi \overline{r}_h)B_h$, where $\overline{r}_h= \sum_{j=1}^s l_j \beta_{j,h} \in \ZZ$. If we pose $r_h = \overline{r}_h$ for $h \ge 2$, $r_1 = \overline{r}_1$ when $\overline{\theta}_1 \ne - \pi$ and $r_1 = \overline{r}_1 -1$ when $\overline{\theta}_1 = - \pi$, we obtain the formula $A= \sum_{h=1}^q (\theta_h+ 2 \pi r_h) B_h$, \,  for some $r_1, \cdots , r_q \in \ZZ$.
\end{proof}

\begin{rem}\label{REMpartskewsymmlogTHM}
When $n=3$, the hypotheses of the previous Theorem are satisfied by every $R \in SO_3 \setminus \{I_3 \}$. This implies that the set of real skew symmetric logarithms of  every such $R$ is countable, discrete and closed in $\mathcal{A}_3$. 

It is also possible to prove that the set of real skew symmetric logarithms of $I_3$ is a countable disjoint union of submanifolds of $\mathcal{A}_3$, all diffeomorphic to $2$-spheres, and of the point consisting in the null matrix. We propose to investigate this subject in the context of a more general study about logarithms of matrices.
\end{rem}

\section{Geodesics on the Riemannian manifold $O_n(\RR)$}

\begin{recalls}\label{trace_metric}

We can consider on $GL_n$ the classical Riemannian metric (independent of the base point), called \emph{Frobenius metric}:
$$
g(V, W) = tr(V^T W).
$$

On the other hand in \cite{DoPe2015} we have considered the semi-Riemannian metric on $GL_n$, given by
$$\overline{g}_G(V,W) = tr(G^{-1}V G^{-1}W)$$

which turns out to be bi-invariant on $GL_n$.

This last metric induces a structure of symmetric semi-Riemannian manifold on $GL_n$. In particular we have described its \emph{Levi-Civita connection} $\nabla$, its \emph{curvature tensors} and characterized its \emph{geodesics} as the curves of the type: 
$$\alpha(t)=Kexp(tC)$$ 
for any $C \in M_n$ and any $K \in GL_n$ (\cite{DoPe2015} Thm. 2.1).

The metric $\overline{g}$ is an important object of study on the submanifold of $GL_n$ of real positive definite matrices, where it defines  a Riemannian structure (see for instance \cite{Lan1999} Chapt.\,XII, \cite{BhaH2006} \S2, \cite{Bha2007} Chapt.\,6, \cite{MoZ2011} \S3).

The metric $\overline{g}$ defines a structure of \emph{anti-Riemannian manifold} on $O_n$ (i.e. it is negative definite at every point of $O_n$), where it is bi-invariant too.

We denote by $g$ and by $\overline{g}$ also the restrictions of the above metrics to the submanifold $O_n$.

Now  if $G \in O_n$, we have: 

$T_G(O_n) = \{ V \in M_n / G^T V = G^{-1} V \mbox{\ is a skew symmetric matrix} \}$ and so, for every $G \in O_n$ and every $V, W \in T_G(O_n)$, we have: 

$\overline{g}_G(V, W) = tr(G^{-1} V G^{-1}W) = -tr((G^TV)^T G^T W) = -tr(V^T W) = -g(V, W)$.

This allows to state the relation between the two metrics on $O_n$:
\end{recalls}

\begin{lemma}\label{g=-g-segnato}
On $O_n$ we have $g = - \overline{g}$, so $(O_n, g)$ and $(O_n, \overline{g})$ have the same geodesics and also $g$ is a bi-invariant metric on $O_n$.
\end{lemma}

\begin{prop}\label{conseguenze_bi-invarianza}
a) $(O_n, g)$ is a globally symmetric Riemannian manifold with non-negative sectional curvature.

b) The geodesic curves of $(O_n, g)$ are precisely the curves of type $P(t) = Gexp(tA)$ with $G \in O_n$ and $A \in \mathcal{A}_n$.

c) $(O_n, g)$ is an Einstein manifold with Ricci tensor $Ric =  \dfrac{n-2}{4} g$. 
Moreover for the scalar curvature $S$ we have $S =   \dfrac{(n-2)(n-1)n}{8}$ .
\end{prop}

\begin{proof}
Assertions (a) and (b) are standard consequences of the bi-invariance of the metric $g$ on $O_n$ (see for instance \cite{Miln1976} \S 7 and \cite{Spiv1979} pp.544--545).

For (c), again by the bi-invariance, we have: $Ric(X,Y) = - B(X,Y)/4$ where $B$ is the Cartan-Killing form of $O_n$ (see for instance \cite{Miln1976} p. 324). Now it is known that, if $X, Y \in \mathcal{A}_n$, then $B(X,Y) = (n-2) tr(XY) = - (n-2) tr(X^T Y)$ (see for instance \cite{Sepa2007} Chapt.\,6, \S 2), so  $B(X,Y) = - (n-2)g(X,Y)$. Therefore $Ric(X, Y) = (n-2) g(X,Y) /4$.
Finally  denoted by $S$ the scalar curvature of $(O_n, g)$, we get that $S=  (n-2)dim(O_n) /4 =   (n-2)(n-1)n/8$.
\end{proof}

\begin{rem}\label{rivestimento_universale}
It is known that a $3$-dimensional Einstein manifold has constant sectional curvature equal to $S/6$ (see for instance \cite{KoNo1963} p.293 Prop.2). Then, from Proposition \ref{conseguenze_bi-invarianza} (c), we get that $(O_3,g)$ has constant sectional curvature $1/8$, hence $(SO_3, g/8)$ is a connected Riemannian manifold with constant sectional curvature equal to $1$ and therefore it is isometric to $\PP^3(\RR)$ with the Fubini-Study metric (see \cite{KoNo1963} Chapt. VI, Thm. 7.10).

\medskip

The following Proposition gives some properties of $O_n$, viewed as submanifold of the semi-Riemannian manifold $(GL_n, \overline{g})$.
\end{rem}
 
\begin{prop}\label{foliations}
For every symmetric positive definite matrix $P$ let 

$\mathcal{L}_P = \{ X \in GL_n \ / \  XX^T = P \} $ and $\mathcal{R}_P = \{ X \in GL_n \ / \  X^TX = P \}$. 

Then 
$
GL_n = \cup_P \mathcal{L}_P =  \cup_P \mathcal{R}_P 
$

(where the unions are taken in the set of all symmetric positive definite matrices) are two foliations of $GL_n$, whose leaves are  totally geodesic anti-Riemannian submanifolds of $(GL_n, \overline{g})$, all isometric to $(O_n, \overline{g})$.
\end{prop}

\begin{proof}
First of all we prove that $(O_n, \overline{g})$ is a totally geodesic anti-Riemannian submanifolds of $(GL_n, \overline{g})$.

Indeed let $G \in O_n$ and $V \in T_G (O_n)$. The geodesic on $(GL_n, \overline{g})$ through $G$ with velocity $V$ is $\alpha(t) = G exp(tG^{-1}V)$. But $tG^{-1}V \in \mathcal{A}_n$ for every $t$, hence $\alpha(t) \in O_n$ for every $t$ and this is equivalent to say that $(O_n, \overline{g})$ is totally geodesic in $(GL_n, \overline{g})$ (see for instance \cite{Lan1999} Chapt.XIV Cor.1.4).

For every fixed symmetric positive definite matrix $P$, we have $\mathcal{L}_P \ne \emptyset$. 
Hence if $Q \in \mathcal{L}_P$, then standard computations show that the left translation $L_Q$ is an isometry of $(GL_n, \overline{g})$ onto itself, mapping $O_n$ onto $\mathcal{L}_P$.

Since $(O_n, \overline{g})$ is a totally geodesic submanifold of $(GL_n, \overline{g})$ of dimension 

$n(n-1)/2$, 
it follows that $\mathcal{L}_P$ is a totally geodesic submanifold of $(GL_n, \overline{g})$ of the same dimension. Moreover $L_Q|_{O_n}: (O_n, \overline{g}) \to (\mathcal{L}_P, \overline{g})$ is also an isometry. Furthermore every $X \in GL_n$ belongs to the submanifold $\mathcal{L}_{XX^T}$, because $XX^T$ is a symmetric positive definite matrix. Therefore $GL_n = \cup_P \mathcal{L}_P$ is a foliation of $(GL_n,\overline{g})$ as predicted. We can conclude the proof, because the transposition is an isometry of $(GL_n, \overline{g})$ onto itself mapping $\mathcal{L}_P$ onto $\mathcal{R}_P$ for every positive definite matrix $P$ (see \cite{DoPe2015} Prop. 1.2).
\end{proof}

\begin{defi}
Let $G, H$ be in the same connected component of $O_n$ \\(so $G^{-1}H \in SO_n$). We call \emph{geodesic arc joining} $G$ \emph{and} $H$ \emph{in} $O_n$ any geodesic curve \\ $\gamma: [0,1] \to (O_n, g)$ with $\gamma(0) = G$, $\gamma(1) = H$. 
\end{defi}

\begin{rem}
As in \cite{DoPe2015} Cor. 2.2, from Proposition \ref{conseguenze_bi-invarianza} (b) we can get that there is a bijection between the geodesic arcs connecting $G, H$ and the real skew symmetric solutions of the matrix equation $exp(X) = G^{-1}H$, i.e. the real skew symmetric logarithms of the matrix $G^{-1} H \in SO_n$.

From Theorem \ref{partskewsymmlogTHM}, we get the following 
\end{rem}

\begin{prop}\label{countable-geodesics}
Let $G, H$ be in the same connected component of $O_n$ and such that all the following conditions on $G^{-1}H$ hold:

- its nonreal complex eigenvalues are all of multiplicity $1$;

- either $-1$ is not an eigenvalue or it has multiplicity $2$;

- either $1$ is not an eigenvalue or it has multiplicity at most $2$.

Then the family of geodesic arcs joining $G$ and $H$ is countable.

\end{prop}

\begin{rem}\label{princ_geod}
Let $\pm \mathbf{i} \zeta_1 , \cdots , \pm \mathbf{i} \zeta_s$ (with $\zeta_1 , \cdots , \zeta_s >0$ and $\zeta_j \ne \zeta_h$ for $j \ne h$) be the distinct nonzero eigenvalues of a given real skew symmetric matrix $A$. We know that there exists an SVD system of skew symmetric matrices $A_1 , \cdots , A_s$ such that $A = \zeta_1 A_1 + \cdots + \zeta_s A_s$, and, moreover, $exp(A) = I_n + \sum_{j = 1}^s[\sin\zeta_j A_j + (1-\cos\zeta_j)A_j^2]$ (remember Theorem \ref{SVD}).

Hence $tA=t \zeta_1 A_1 + \cdots + t \zeta_s A_s$, so we get that the generic geodesic $\alpha(t)$ of $(O_n, g)$  with $\alpha(0) = G$ can be expressed as 
$$
\alpha(t) = G exp(tA) = G + \sum_{j = 1}^s[\sin(t\zeta_j) GA_j + [1-\cos(t\zeta_j)]GA_j^2]
$$
where $s \le \lfloor n/2 \rfloor$.

If $\alpha(t)$ is a geodesic arc joining $G, H \in O_n$, then $A$ is a skew symmetric logarithm of $G^{-1}H$ . Hence, by Theorem \ref{all_skew_symmetric_logarithms}, there exist a real skew symmetric principal logarithm $B$ of $G^{-1} H$, an SVD system of skew symmetric matrices $C_1, \cdots , C_s$  all commuting with $B$ and $l_1, \cdots , l_s \in \ZZ$ such that
$
A = B + 2\pi \sum_{j= 1}^s l_j C_j.
$

Thus the geodesic arc  $\alpha(t)$ is
\begin{multline*}
\alpha(t) = G exp(tB)\{I_n + \sum_{j=1}^s(\sin(2 \pi t l_j)C_j + (1-\cos(2 \pi t l_j))C_j^2) \} = \\ 
G exp(tB) + Gexp(tB)\{\sum_{j=1}^s(\sin(2 \pi t l_j)C_j + (1-\cos(2 \pi t l_j))C_j^2)\} = \alpha_{princ}(t) + \alpha_{var}(t)
\end{multline*}
where $\alpha_{princ}(t)=G exp(tB)$ is still a geodesic arc joining $G, H$ corresponding to the real skew symmetric principal logarithm B of $G^{-1}H$, we call \emph{principal geodesic arc}, while the \emph{variation}, $\alpha_{var}(t)$, is a loop in $M_n$ at the null matrix.
\end{rem}

\begin{prop}\label{rat_numb_SO_n}
With the same notations as in Remark \ref{princ_geod}, let 

$\alpha(t) = G exp(tA) = G + \sum_{j = 1}^s[\sin(t\zeta_j) GA_j + [1-\cos(t\zeta_j)]GA_j^2]$ be a nonconstant geodesic of $(O_n, g)$. Then

a) if there exist $j, h \in \{1, \cdots, s\}$ such that $\dfrac{\zeta_j}{\zeta_h} \notin \QQ$, then $t \mapsto \alpha(t)$ is an injective immersion of $\RR$ into $O_n$;

b) if \ \ $\dfrac{\zeta_j}{\zeta_h} \in \QQ$ for every $j, h \in \{1, \cdots, s\}$, then $\alpha(t)$ is a closed (i.e. periodic) geodesic of $(O_n, g)$.
\end{prop}
 
\begin{proof}
Suppose that $\alpha(t)=\alpha(t')$. This gives

$
\sum_{j = 1}^s[\sin(t\zeta_j)A_j +(1-\cos(t\zeta_j)) A_j^2] =
\sum_{j = 1}^s[\sin(t'\zeta_j)A_j +(1-\cos(t'\zeta_j)) A_j^2] 
$.

Multiplying by $A_k$ (for $k= 1, \cdots , s$) the above equality and comparing the symmetric and skew symmetric parts, we get

$\sin(t\zeta_k) = \sin(t'\zeta_k)$ and $\cos(t\zeta_k) = \cos(t'\zeta_k)$ for every $k$ and so  standard arguments give $\dfrac{\zeta_j}{\zeta_h} \in \QQ$ for every $j, h \in \{1, \cdots , s \}$. This concludes (a).

Assume that $\zeta_j / \zeta_h \in \QQ$ for any $j, h= 1, \cdots , s$.  After reducing $\zeta_j / \zeta_1$, for $j = 1, \cdots , s$, to a common denominator, we get: $\zeta_j / \zeta_1 = r_j / m$, with $r_j, m \in \NN \setminus \{0\}$. 

If $T= 2 \pi m / \zeta_1$, it is easy to check $\alpha(t+T) = \alpha(t)$
for any $t \in \RR$. This gives (b).
\end{proof}

\begin{thm}\label{minimalgeodesics}
Let $G, H$ be in the same connected component of $O_n$. The curves of minimal length among the curves in $(O_n, g)$, joining $G$ and $H$, are precisely the principal geodesic arcs i.e. the curves
$\alpha: [0,1] \to O_n$ of type $\alpha(t) = Gexp(tB)$, where $B$ is any skew symmetric principal logarithm of $G^{-1}H$.

Moreover, denoted by $d$ the distance associated to the metric $g$, we have $d(G,H) = \sqrt{\sum_{j=1}^n |\log \mu_j|^2}$, where $\mu_1, \cdots , \mu_n$ 
are the (possibly repeated) $n$ eigenvalues  of $G^{-1}H$ and $\log \mu$ denotes the principal complex logarithm of the complex number $\mu$.
\end{thm}

\begin{proof}
Since $(SO_n, g)$ and $(O_n^-,g)$ are geodesically complete, by Hopf-Rinow theorem (see for instance \cite{O'N1983} Chapt.\,5, 21--22) there are curves of minimal length among the curves joining $G, H$ (belonging to the same connected component of $O_n$); they are geodesic arcs joining $G, H$  and have length equal to the Riemannian distance $d(G,H)$.  
A geodesic arc $\alpha(t)$ joining $G, H$  is a curve of type $\alpha(t) = G exp(tA)$, with $A$ real skew symmetric logarithm of $G^{-1} H \in SO_n$.

By Remark \ref{princ_geod}, $\alpha(t) =  \alpha_{princ}(t) + \alpha_{var}(t)$, where $ \alpha_{princ}(t) = G exp(tB)$ and $B$ is a real skew symmetric principal logarithm of $G^{-1} H$. Since $g(\stackrel{.}{\alpha}(t), \stackrel{.}{\alpha}(t))$ is constant, by Proposition \ref{inequality} we have 

 \begin{multline*}
length(\alpha)= \int_0^1 \sqrt{g(\stackrel{.}{\alpha}(t), \stackrel{.}{\alpha}(t))}dt=
\sqrt{g(\stackrel{.}{\alpha}(0), \stackrel{.}{\alpha}(0))} = 
\\
\sqrt{tr(A^T A}) = \sqrt{-tr(A^2)} \ge \sqrt{-tr(B^2)},
\end{multline*}

with equality if and only if $A$ is a real skew symmetric principal logarithm, i.e. if and only if $\alpha(t)$ is a principal geodesic arc. This gives the first part of the statement. 

Finally, by \ref{eigenvalues_B}, $d(G,H) = \sqrt{-tr(B^2)} = \sqrt{\sum_{j=1}^n |\lambda_j|^2}$, where $\lambda_1 , \cdots, \lambda_n$ are the (possibly repeated) eigenvalues of any real skew symmetric principal logarithm of $G^{-1}H$, hence we can conclude because $|\lambda_j| =| \log \mu_j |$ for every $j$.
\end{proof}

\begin{cor}\label{misura-diametro}
The diameter of $(O_n, g)$ is equal to $\sqrt{2 \lfloor n/2 \rfloor} \pi$.
\end{cor}

\begin{proof}
By definition the diameter is 

$diam(O_n, g)=
sup \{d(G,H) \, / \, G,H \mbox{ in the same connected component of } O_n \}$. 

So the result follows directly from Theorem \ref{minimalgeodesics}.
\end{proof}

\begin{rem}\label{minimal_length}
The geodesics of minimal length among curves in $O_n$ joining $G$ and $H$ are all periodic or are all injective immersions $\alpha: \RR \to O_n$.

This follows from the Proposition \ref{rat_numb_SO_n} and from the fact that different skew symmetric principal logarithms have always the same eigenvalues.

By Theorem \ref{Characterize_plog}, we get that if $-1$ is not an eigenvalue of $G^{-1} H$, then the geodesic of minimal length joining $G$ and $H$ (belonging to the same connected component of $O_n$) is unique, while if $-1$ is an eigenvalue of $G^{-1} H$ with (even) multiplicity $2m_1$, then the geodesics of minimal length joining $G$ and $H$ (both in $SO_n$ or in $O_n^-$ as above) can be parametrized by means of the manifold $\mathcal{M}_{2m_1}$ diffeomorphic to $\mathcal{A}plog(G^{-1} H)$.
\end{rem}

\begin{prop}\label{grassmannian}
Let $\mathcal{N}_{p, n-p}$ be the set of real symmetric orthogonal matrices of order $n$ with the eigenvalues $1$ of multiplicity $p$ and $-1$ of multiplicity $n-p$. Then $\mathcal{N}_{p, n-p}$ is a submanifold of the vector space $S_n$ of real symmetric matrices of order $n$, diffeomorphic to the Grassmannian $Gr(p, n)$ of $p$-dimensional vector subspaces of $\RR^n$. In particular $\mathcal{N}_{1, n-1}$ is diffeomorphic to $\PP^{n-1}(\RR)$.
\end{prop}

\begin{proof}
Analogously to Remark \ref{OSSMatrices} and Proposition \ref{class_with_Pfaff}, we denote by $\phi (Q, N) = Q N Q^T$ the left action of $O_n$ on $S_n$. By spectral theorem, $\mathcal{N}_{p, n-p}$ is the orbit of 
$diag(I_p, -I_{n-p})$. 
The isotropy group of $diag(I_p , - I_{n-p})$ consists of orthogonal matrices commuting with this matrix and, therefore, of block diagonal matrices belonging to $O_p \times O_{n-p}=\{ diag(U,U')\ / \ U \in O_p,\  U' \in O_{n-p} \}$. Again $\mathcal{N}_{p, n-p}$ is a differentiable submanifold of $S_n$ diffeomorphic to $O_n / (O_p \times O_{n-p})$ (see for instance \cite{Popov1991} p.1). We conclude, since this quotient is the expected Grassmannian (see for instance \cite{BrtD1985} I, 4.12).
\end{proof}

\begin{defis}\label{weakly_diametral}
Let $G, H$ be in the same connected component of $O_n$. We say that they form 

a) a \emph{weakly diametral pair} if there is a geodesic curve $\gamma$ such that $\gamma(0)= G$, $\gamma(1) = \gamma(-1) = H$,

b) a \emph{diametral pair} if $d(G,H) = diam(O_n, g) = \sqrt{2 \lfloor n/2 \rfloor} \pi$.
\end{defis}

\begin{rem}\label{diameter}
If $A \in \mathcal{A}_n$, then $[exp(A)]^T = exp(-A)$. So it is easy to show that the matrices $G, H$ (as above) form a weakly diametral pair if and only if $G^{-1}H$ is symmetric.
Hence $G, H$ form a weakly diametral pair if and only if $G^{-1}H \in \mathcal{N}_{n-2q, 2q}$ for some $q=0, \cdots , \lfloor n/2 \rfloor$, i.e. if and only if $H \in \cup_{q = 0}^{\lfloor n/2 \rfloor} L_G(\mathcal{N}_{n-2q, 2q})$ ($L_G$ denotes the left translation) or equivalently, by Theorem \ref{Characterize_plog}, if and only if $\mathcal{A}plog(G^{-1}H)$ is symmetric with respect to the null matrix. Proposition \ref{grassmannian} allows to obtain the following 
\end{rem}

\begin{prop}\label{weakly_diametral_pairs}
Let $\mathcal{N}_{n-2q, 2q}$ be as in Proposition \ref{grassmannian} and $G \in O_n$. Then the set of $H \in O_n$ such that $G, H$ form a weakly diametral pair, is the disjoint union $\cup_{q = 0}^{\lfloor n/2 \rfloor} L_G(\mathcal{N}_{n-2q, 2q})$, where every $L_G(\mathcal{N}_{n-2q, 2q})$ is a submanifold of $O_n$, diffeomorphic to $Gr(n-2q, n)$.
\end{prop}

\begin{prop}\label{diametral_pairs}
Let $G, H$ be in the same connected component of $O_n$. Then they form a diametral pair if and only if $H \in L_G(\mathcal{N}_{n-2\lfloor n/2 \rfloor, 2\lfloor n/2 \rfloor})$; hence a diametral pair is weakly diametral too.

In particular $G, H$ is a diametral pair if and only if $H=-G$ in case of $n$ even and  if and only if $H \in L_G(\mathcal{N}_{1, n-1})$ in case of $n$ odd. In this last case the set of points $H$ such that $G,H$ is a diametral pair is a submanifold of $O_n$ diffeomorphic to $\PP^{n-1}(\RR)$.
\end{prop}

\begin{proof}
The first characterization follows from Theorem \ref{minimalgeodesics}. We can conclude by means of the description of $\mathcal{N}_{1, n-1}$ stated in Proposition \ref{grassmannian}.
\end{proof}

\end{document}